\documentclass[10pt]{article}
\usepackage{amsmath}
\usepackage{amsthm}
\usepackage{amsfonts}
\usepackage{amssymb}
\usepackage{latexsym} 

\usepackage[matrix,tips,graph,curve]{xy}



\makeatletter 
\@addtoreset{equation}{section}
\makeatother

\theoremstyle{plain}
\newtheorem{theorem}[equation]{Theorem}

\newtheorem{proposition}[equation]{Proposition}
\newtheorem{conjecture}[equation]{Conjecture}

\theoremstyle{definition}
\newtheorem{define}[equation]{Definition}

\newcommand{\tr}{\mathrm{tr}}

\renewcommand\dim{{\rm dim\,}}

\def\d/{/\mspace{-6.0mu}/}

\newcommand{\p}{\partial}

\begin{document}

\title{Nonlinear harmonic forms and Indefinite Bochner Formulas}
\author{Mark Stern}
\footnotetext{Duke University, Department of Mathematics;  
e-mail:  stern@math.duke.edu,\\ partially supported by by NSF grant DMS 1005761}
\date{}

 \maketitle
 
\section{Introduction}
The link between topology and geometry provided by harmonic forms and $L_2$-cohomology has played a fundamentally important role in differential and algebraic geometry. In this note, we introduce a new object to geometric analysis, which we call a {\em nonlinear harmonic form}.  We define these forms as closed differential forms $z$ which 
\begin{itemize}\item represent a fixed class in de Rham cohomology, and
\item minimize $\|z\|^2_{L_2}$ subject to a natural nonlinear constraint $p(z) = 0.$
\end{itemize}
Because the energy is the usual $L_2$ energy, the nonlinear harmonic forms satisfy rich Euler Lagrange equations. 
The constraints we consider in this paper are polynomial. Our initial regularity results require that $p$ is diffeomorphism invariant.  Our first example, which arises from our investigations into the Hopf conjecture, imposes the constraint $p(z) = z\wedge z$. These forms are in some ways intermediate between surfaces and harmonic forms, and their analysis is similar to that of harmonic maps. It is not difficult to prove minimizers exist, but their regularity theory remains to be developed. For example, we have the following theorem. 
\begin{theorem}\label{exists}Let $M^n$ be a compact Riemannian manifold. Let $f $ be a $p-$form with 
$f\wedge f = 0.$ Let $h $ be the harmonic representative of $f$. 
Set $$Q:=\{y \in H_1 :(h+dy)\wedge (h+dy)  = 0, \,\, \text{ and }d^*y =0\}.$$ 
Let $E(y) = \|h+dy\|^2 $. Let $\nu_0:= \inf\{E(y):y\in Q\}.$
Then there exists $y\in H_1$ such that $E(y) = \nu_0.$ Moreover $z:=h+dy$ lies in the Morrey space $L^{2,n-2p}$. 
\end{theorem}
  In this note, we establish only the most elementary properties of nonlinear harmonic forms. A careful analysis of their singularities requires further study. 

Our interest in these nonlinear harmonic forms was first sparked by our application of a new indefinite Bochner formula to  investigation of the Hopf conjecture. Although we have not fruitfully applied nonlinear harmonics to this conjecture, we include a discussion of the Hopf conjecture, which may be of independent interest.

\begin{conjecture}\label{conjh}(Hopf) 
$S^2\times S^2$ does not admit a metric of strictly positive sectional curvature.  
\end{conjecture}

In order to apply harmonic form techniques to this question, we consider a stronger conjecture. 
\begin{conjecture}\label{conj} A compact oriented $4-$manifold with positive sectional curvature with nonvanishing second betti number has definite intersection form. 
\end{conjecture}
As $H^2(S^2\times S^2)$ has an indefinite intersection form, the Hopf conjecture follows immediately from Conjecture \ref{conj}. 

Let $h$ be a harmonic 2- form on a compact oriented Riemannian 4-manifold of positive sectional curvature. We may decompose $h$ as
$$h= h_++h_-,$$
where 
$$\ast h_{\pm} = \pm h_{\pm}.$$ Conjecture \ref{conj} implies that 
\begin{equation}\label{0prod}|h_+||h_-| \equiv 0.\end{equation}
  In a local orthonormal frame, the Bochner formula gives 
\begin{equation}\label{bochnerintro}-\frac{1}{2}\Delta|h|^2 =
  |\nabla h|^2 -  R_{ijkl}   h_{kj} h_{il}  +   \text{Ric}_{il} h_{lq}h_{iq} .  \end{equation}
Positivity of the sectional curvature implies positivity of the Ricci term in (\ref{bochnerintro}) but does not imply positivity of the remaining curvature term without additional assumptions.  
Equation (\ref{0prod}) suggests we replace $|h|^2$ in the the Bochner formula with $|h_+|^2|h_-|^2$. This leads to an indefinite Bochner formula.

  In a neighborhood of a point where $|h_+|^2|h_-|^2\not = 0$, we can choose an oriented orthonormal coframe $\{w^i\}_{i=1}^4$ and frame $\{e_i\}_{i=1}^4$ in which $h$ takes the form 
$h = aw^1\wedge w^2+bw^3\wedge w^4$, with $a>b$. In this frame, we have (see (\ref{exboch})) 
\begin{equation}\label{exoboch}-\frac{1}{4}\Delta\ln(|h_+|^2|h_-|^2)  =  R_{1331}+R_{2442}+R_{2332}+R_{1441}+|II|^2-\frac{1}{2}|T|^2, \end{equation}
where $II$ is a generalized second fundamental form and $T$ is a vectorfield which vanishes if and only if the distribution spanned by $\{e_1,e_2\}$ and the distribution spanned by $\{e_3,e_4\}$ are both integrable. 
 
The only curvature terms which appear in (\ref{exoboch})  are sectional curvatures, making such exotic Bochner formulas appear well suited to studying metrics of positive sectional curvature, if we can control the torsion terms. In order to eliminate half the torsion terms, we study nonlinear harmonic forms subject to the constraint $z\wedge z=0$. For such forms, we obtain the integrability of the distribution spanned by $\{e_3,e_4\}$, eliminating a summand of $T$. Moreover, the Euler Lagrange equations for the nonlinear harmonic still imply equation (\ref{exoboch}). In fact, these equations can be derived (\ref{fp1vareq}) using only the fact that $0$ is a critical point of  $\|\phi_t^*h\|^2$, for all smooth  one parameter families of diffeomorphisms satisfying $\phi_0(x) = x$. This does not suffice to settle (\ref{conj}) but suggests that the nonlinear harmonics deserve careful study. 
 
I would like to thank Daniel Stern for explaining the significance of Bettiol's work and for many useful discussions. I also thank Hugh Bray and Robert Bryant for helpful discussions. 

\section{An indefinite Bochner formula}

Let $h$ be a strongly harmonic 2 form on a riemannian $4-$manifold $M$. In a neighborhood of a point where $h$ is neither self dual nor antiself dual, we may choose an oriented local orthonormal frame $\{e_i\}_{i=1}^4$ and dual frame $\{w^i\}_{i=1}^4$ so that 
\begin{equation}h = aw^1\wedge w^2 + bw^3\wedge w^4.\end{equation}
Let 
$\Sigma_a$ and $\Sigma_b$ denote the distributions spanned by $\{e_1,e_2\}$ and $\{e_3,e_4\}$ respectively. 
Write 
\begin{equation}\nabla_{e_i}e_j = \gamma_{ij}^ke_k, \text{ and }c^k_{ij}:= \gamma_{ij}^k - \gamma_{ji}^k.\end{equation}
Then the torsion vectors 
\begin{equation}T_a:= c^3_{12}e_3+c^4_{12}e_4\text{ and }T_b:=c^1_{34}e_1+c^2_{34}e_2\end{equation}
vanish if and only if the respective distributions $\Sigma_a$ and $\Sigma_b$ are integrable. Define the generalized second fundamental form of $\Sigma_a$ and $\Sigma_b$ by 
\begin{equation}
II_a:=  \sum_{i,j =1,2}\sum_{k=3,4}\frac{1}{2}(\gamma^m_{ij}+\gamma^m_{ji})w^i\otimes w^j\otimes e_m,
\end{equation}
and 
\begin{equation}
II_b:=  \sum_{i,j =3,4}\sum_{k=1,2}\frac{1}{2}(\gamma^m_{ij}+\gamma^m_{ji})w^i\otimes w^j\otimes e_m.
\end{equation}
The generalized mean curvature vectors of these second fundamental forms are  given by 
\begin{equation}H_a = (\gamma^3_{11}+\gamma^3_{22})e_3 +  (\gamma^4_{11}+\gamma^4_{22})e_4,\end{equation}
and 
\begin{equation}H_b = (\gamma^1_{33}+\gamma^1_{44})e_1 + (\gamma^2_{33}+\gamma^2_{44})e_2.\end{equation}
In this notation, $dh = 0$ gives 
$$0 = a_3 - aH^3_a + bT_a^4 ,$$
$$0 = a_4 - aH^4_a -bT_a^3 ,$$
$$0 = b_1 - bH_b^1+ aT_b^2,$$
$$0 = b_2 - bH_b^2- aT_b^1,$$
and $d^*h= 0$ gives 
$$0 = a_1 - aH_b^1+ bT^2_{b}  ,$$
$$0 = a_2  - aH_b^2- bT^1_{b} ,$$
$$0 = b_3 - bH_a^3 + aT^4_{a} ,$$
$$0 = b_4 - bH^4_a - aT^3_{a} .$$
Setting $f = \ln(\sqrt{a^2-b^2}) $, gives 
\begin{equation}\label{1hflow}\nabla f = H_a+H_b;\end{equation}
so $f$ is a generalized area function. Taking second derivatives, we have 
$$  f_{11}  =   \gamma^1_{33,1}+\gamma^1_{44,1}  .$$
$$  f_{22}  =   \gamma^2_{33,2}+\gamma^2_{44,2} . $$
$$  f_{33} =  \gamma^3_{11,3}+\gamma^3_{22,3} .$$
$$ f_{44} =  \gamma^4_{11,4}+\gamma^4_{22,4} .$$

Hence
\begin{equation}\label{preboch}\sum_j f_{jj} =  \gamma^1_{33,1}+\gamma^1_{44,1} + \gamma^2_{33,2}+\gamma^2_{44,2}  +  \gamma^3_{11,3}+\gamma^3_{22,3}   + \gamma^4_{11,4}+\gamma^4_{22,4}.\end{equation}
In this frame, the sectional curvatures are given by
\begin{equation}\label{gauss}R_{ijji} = \gamma_{jj,i}^i + \gamma_{im}^i\gamma_{jj}^m - \gamma_{ij,j}^i - \gamma_{jm}^i\gamma_{ij}^m - (\gamma_{ij}^m-\gamma_{ji}^m)\gamma_{mj}^i.\end{equation}
 At a fixed  point, we may rotate our frame so that $\gamma_{i3}^4 = 0 = \gamma_{i1}^2$. In such a frame adapted to the given point, we have 
$$R_{1331} +  (\gamma_{11}^3)^2+  (\frac{\gamma_{12}^3 + \gamma_{21}^3}{2})^2 - (\frac{ c_{12}^3}{2})^2 
+(\gamma_{33}^1 )^2+  (\frac{\gamma_{34}^1 + \gamma_{43}^1}{2})^2 - (\frac{ c_{34}^1}{2})^2   = \gamma_{33,1}^1 + \gamma_{11,3}^3 .$$
Inserting this and related equations for other sectional curvatures into (\ref{preboch}) yields our indefinite Bochner formula. At a critical point,
\begin{equation}\label{exboch}
 -\Delta f  =  R_{1331}+
R_{2332}+
R_{1441} +R_{2442}   + |II_a|^2+|II_b|^2-\frac{1}{2}|T_a|^2 -\frac{1}{2}|T_b|^2 
 . \end{equation}
If the torsion vectors in the above expressions vanished or were dominated by the second fundamental form terms, Conjecture \ref{conj} would follow immediately from (\ref{exboch}).  
Wherever $h$ is neither self dual nor antiself dual, we have 
\begin{equation}\label{exbochcurv}
 -\Delta f =  \kappa   +  R_{1221}+ R_{3443} -2R_{1221}^{\Sigma_a}-2R_{3443}^{\Sigma_b}-\frac{1}{2}|T_a|^2 -\frac{1}{2}|T_b|^2 
 , \end{equation}
where $\kappa$ denotes the scalar curvature of $M$, and $R^{\Sigma_a}$ and $R^{\Sigma_b}$ denote the Riemann curvature of the leaves of the respective foliation when the distributions are integrable. Otherwise they are defined by the Gauss equations in terms of the second fundamental form and the curvature of $M$ by the same relations as the integrable case. 

The likelihood of applying these formula to the Hopf conjecture is somewhat diminished by results of Renato Bettiol. 
\begin{define}Let $\{v_j\}_{j=1}^4$ be any orthonormal basis of $T_pM$, $M$ a $4-$manifold. \\
$\frac{1}{2}(R(v_1,v_3,v_3,v_1) + R(v_2,v_4,v_4,v_2))$ is called the {\em biorthogonal curvature} associated to the plane $\sigma$ spanned by $\{v_1,v_3\}$ and its orthogonal complement $\sigma^\perp$. \end{define}
 Bettiol \cite{Bettiol} has proved that $S^2\times S^2$ admits metrics with positive biorthogonal curvature. The curvature terms appearing in (\ref{exoboch})  occur in biorthogonal pairs, and therefore more information is needed to apply this formula to Conjectures \ref{conjh} or \ref{conj}.  


 We now introduce a special case of nonlinear harmonic forms in order  to remove some of the torsion terms in  (\ref{exboch}). 
\section{Heuristic Properties of nonlinear harmonic forms}\label{gframe}

In this section, in order to anticipate the fruits of future regularity analysis,  we formally investigate properties of minimizers of $\|z\|^2$ in a fixed degree 2 cohomology class on a compact 4-manifold, subject to the constraint that $z\wedge z = 0$. Again, we emphasize that the discussion in this section  is heuristic only. In subsequent sections we prove several of these heuristic results.

Let $f$ be a closed $2$-form satisfying $f\wedge f=0$. 
Suppose we consider closed 2-forms $z$ cohomologous to $f$ such that $\|z\|_{L_2}^2$ is minimal subject to the constraint $z\wedge z=0.$ 
Then formally 
$\langle d^*z,b\rangle_{L_2} = 0$ for all $b$ satisfying $z\wedge db = 0$. Hence $d^*z$ is perpendicular to the kernel of $e(z)d$, 
where $e(\phi)$ denotes exterior multiplication on the left by $\phi$. Thus $d^*z$ is in the image of $d^*e^*(z)$. That is 
\begin{equation}d(\ast z -\mu z)= 0,\end{equation}
for some function $\mu$, which may be viewed formally as a Lagrange multiplier. 
Where $z\not = 0,$ let $\{e_i\}_{i=1}^4$ and $\{w^j\}_{i=1}^4$ be an oriented orthonormal frame and dual frame so that 
 $z = aw^1\wedge w^2$, with $a> 0$. Then $dz = 0$ implies 
\begin{equation}\label{mu3}0 = a_3-a(\gamma^3_{11} + \gamma^3_{22}),\end{equation}
\begin{equation}\label{mu4}0 = a_4-a(\gamma^4_{11} + \gamma^4_{22}),\end{equation}
\begin{equation}\label{mu5}0 = ac^1_{34} = ac^2_{34}.\end{equation}
So, $\{e_3,e_4\}$ span an integrable distribution. The constraint on $d^*z$ gives 
 \begin{equation}\label{mu1}0 =    a (c^3_{12} -\mu_4),\end{equation}
 \begin{equation}\label{mu2}0 =   a(c^4_{12} +\mu_3),\end{equation}
 \begin{equation}\label{mu6}0 = a_1-a (\gamma^1_{33}+\gamma^1_{44}), \end{equation}
 \begin{equation}\label{mu7}0 =  a_2-a (\gamma^2_{33}+\gamma^2_{44}).  \end{equation}
 Equations (\ref{mu3}), (\ref{mu4}), (\ref{mu6}), and (\ref{mu7}) can again be expressed more succinctly in the notation of (\ref{1hflow}) as
\begin{equation}\label{hflow}\nabla\ln(a) = H_a+H_b.\end{equation} 
Equations (\ref{mu1}) and (\ref{mu2}) imply 
\begin{equation}\label{Tmu}aT\mu = 0,\end{equation}
where $$T:= c_{12}^3e_3+c_{12}^4e_4.$$
 Equations (\ref{mu1}) and (\ref{mu2}) also imply that we may view $T$ as the Hamiltonian vector field for $\mu$ restricted to the leaves of the foliation associated to the integrable distribution spanned by $\{e_3,e_4\}$, equipped with the symplectic form $ w^3\wedge  w^4$ pulled back to the leaves. Let $\phi_s$ denote the one parameter family of diffeomorphisms generated by $aT$ on the space $a>0$. Equations (\ref{Tmu}) implies that $\phi_s^*\mu = \mu$. We also have $\phi_s^*z = z $ since $i_Tz = 0$ and $z$ is closed, where $i_X$ denotes interior product with $X$.  The volume form is also invariant under $\phi_s$ since  
$$d i_{aT} dvol = d(c_{12}^3aw^1\wedge w^2\wedge w^4 -c_{12}^4aw^1\wedge w^2\wedge w^3) 
= d(d\mu\wedge z) = 0.$$  
Hence, the nonlinear harmonic form comes equipped with a foliation with a symplectic structure, and a distinguished Hamiltonian whose Hamiltonian vector field preserves both the nonlinear harmonic form and the volume form. 
 
Returning to Bochner formulas, let now $f = \ln(a)$. At a critical point $p$ of $f$, the mean curvature of the leaf through $p$ is zero. The Bochner formula at the critical point becomes  
\begin{equation}\label{nboch}-\Delta f =   R_{1331}+
R_{2332}+
R_{1441} +R_{2442}  +|II_a|^2+|II_b|^2   - \frac{1}{2}|d\mu|^2. \end{equation}
This is essentially identical to  (\ref{exboch}). 
Once again the curvatures appear in biorthogonal pairs. Let $f$ have a maximum at $p\in M$. Considering the positivity of the full hessian rather than simply its trace does not appear to be useful.

\section{Existence of quadratic nonlinear harmonic forms}

In this section, using standard techniques, we prove the existence of nonlinear harmonic forms in $L_2$ subject to quadratic constraints. 
Let $H_s$ denote the Sobolev space of differential forms $\phi$ with $(1+\Delta)^{s/2}\phi\in L_2$ and norm 
$\|\phi\|_{H_s} = \|(1+\Delta)^{s/2}\phi\|_{L_2}.$   
\begin{theorem}\label{exists}Let $M^n$ be a compact Riemannian manifold. Let $f:=(f_1,\dots,f_r)$ be an $r-$tuple of closed differential forms with $$p^k(f):= \sum_{i,j}p_{ij}^k\wedge f_i\wedge f_j = 0,\,\,1\leq k\leq D,$$ for some smooth differential forms  $p_{ij}^k$. Let $h_i$ denote the harmonic representative of $f_i$. 
Set $$Q:=\{y=(y_1\cdots,y_r)\in H_1 :p^k(h +dy )  = 0, \,\,1\leq k\leq D \text{ and }d^*y_i=0\}.$$ 
Let $E(y) := \|h+dy\|^2:=\sum_i\|h_i+dy_i\|^2$. Let 
 $\nu_0:= \inf\{E(y):y\in Q\}.$ 
Then there exists $y\in H_1$ such that $E(y) = \nu_0.$
\end{theorem}
\begin{proof}
Let $\{y^m\}_{m=1}^\infty\subset Q$ be a minimizing sequence for $E$. Then the sequence is bounded in $H_1 $. By passing to a subsequence, we may assume that 
$y^m\stackrel{H_1}{\rightharpoonup} y$ and $y^m\stackrel{L_2}{\to} y$, as $m\to\infty,$ for some $y\in H_1$. 
 Then for all smooth forms $ \phi$  and  $\forall k$,
$$0 = \lim_{m\to\infty} \int  \phi\wedge p^k(h +dy^m)  =  \int \phi\wedge p^k(h +dy ) + \lim_{m\to\infty}\int \phi \wedge p^k(dy^m -dy ) $$
$$=  \int \phi \wedge p^k(h +dy )- \lim_{m\to\infty}\int (-1)^{deg(\phi\wedge p^k_{ij})}d(\phi \wedge p_{ij}^k)\wedge (y^m_i- y_i)\wedge ( dy^m_j-dy_j)
=  \int \phi \wedge p^k(h +dy ).$$
Hence  $y\in Q.$ 
Next we observe that 
$$\nu_0 = \lim_{m\to\infty}\|h+dy^m\|^2 = \|h+dy \|^2 + \lim_{m\to\infty}\| dy^m-dy\|^2  $$
implies $E(y) = \nu_0$, and $\lim_{m\to\infty}\| dy^m-dy\|^2   =0$. In particular, 
$y^m\stackrel{H_1}{\to} y$. 
\end{proof} 
We call 
\begin{equation}\label{nlh}z:=h+dy
\end{equation}
a nonlinear harmonic representative of $f$. Observe $d^*z\in H_{-1}^1$. 

It is not clear that we may replace the minimizing sequence by a sequence of smooth $y_n$. Smooth minimizing sequences would likely significantly aid regularity theorems. See \cite{MR}. We cannot yet rule out, however, the possibility that the minimal energy for $y\in H_1$ is strictly smaller than the infimum of energies for smooth $y$. In particular, we have not established that smooth forms satisfying the constraint are dense in the set of $H_1$ forms satisfying the constraint. Similar issues arise in the study of harmonic maps.  
\subsection{Price Monotonicity}

Recall the Morrey norm defined by 
\begin{equation}
\|u\|_{L^{p,\mu}}: = (\sup_{x\in M, 0<r<\text{diam}(M)}r^{-\mu}\int_{B_r(x)}| u|^pdv)^{\frac{1}{p}}.
\end{equation}
Then the Morrey space $L^{p,\mu}$ is defined to be the subset of $L^p$ with finite $L^{p,\mu}$ norm. 
 \begin{theorem}
If the nonlinear harmonic form $z$ defined in (\ref{nlh}) is homogeneous of degree $p$, and if the coefficients $p^k_{ij}$ are constants, then $z\in L^{2,n-2p}$. 
\end{theorem}
\begin{proof} 
Let $Y$ be a smooth vectorfield. Let $\phi_t$ be the  one parameter family of diffeomorphisms generated by $Y$. Then $\phi_t$ preserves $H_1$ and since the $p^k_{ij}$ are constant, $p^k(f^*\zeta) = f^*p^k(\zeta) = 0$ if $p^k(\zeta) = 0.$  Hence 
$\|\phi_t^*z\|_{L_2}^2\geq \|z\|^2_{L_2}.$ In particular, (see (\ref{1varf}) or \cite[p. 142]{Price})
\begin{equation}\label{price}0 = \int \langle z,(-\frac{1}{2}div(Y)+Y^i_{;j}e(w^j)i_{e_i})z\rangle.\end{equation}
For smooth $z$, this is equivalent to 
\begin{equation}\label{rescue}\langle d^*z,i_Yz\rangle  = 0.\end{equation}
Equation (\ref{price}) allows us to prove monotonicity formulas in the usual fashion. We include the details for the convenience of the reader. 

Fix geodesic coordinates $\{x^i\}$ centered at some $p\in M$, and for some function $y$ choose 
$$Y := y(\frac{r^2}{2s}) x^j\frac{\p}{\p x^j}. $$
Then 
$$div(Y) = \frac{1}{\sqrt{g}}\frac{\p}{\p x^j}(\sqrt{g}y(\frac{r^2}{2s}) x^j)= 
ny(\frac{r^2}{2s})    + \frac{r^2}{s} y'(\frac{r^2}{2s}) +\frac{r}{2}y(\frac{r^2}{2s})\frac{\p \ln(g)}{\p r}$$
$$= ny(\frac{r^2}{2s})   -2s\frac{d}{ds} y(\frac{r^2}{2s}) +\frac{r}{2}y(\frac{r^2}{2s})\frac{\p \ln(g)}{\p r},$$
and
$$Y^i_{;j}e(w^j)i_{e_i} =  y(\frac{r^2}{2s}) e(dx^j)i_{\frac{\p}{\p x^j}} +  y'(\frac{r^2}{2s})\frac{r^2}{s} e(dr) i_{\frac{\p}{\p r}} + y(\frac{r^2}{2s}) x^m\Gamma_{jm}^ie(dx^j)i_{\frac{\p}{\p x^i}}$$
$$=  y(\frac{r^2}{2s}) e(dx^j)i_{\frac{\p}{\p x^j}} - 2s\frac{d}{ds}  y(\frac{r^2}{2s})  e(dr) i_{\frac{\p}{\p r}} + y(\frac{r^2}{2s}) x^m\Gamma_{jm}^ie(dx^j)i_{\frac{\p}{\p x^i}}$$

Hence 
$$s^{p-1-\frac{n}{2}}(-\frac{1}{2}div(Y)+Y^i_{;j}e(w^j)i_{e_i})z = (  \frac{d}{ds}(s^{p -\frac{n}{2}} y )  - 2s^{p -\frac{n}{2}} \frac{d y}{ds}    e(dr) i_{\frac{\p}{\p r}} + s^{p-1-\frac{n}{2}}y w )z  ,
$$
where $w:=- \frac{r}{4} \frac{\p \ln(g)}{\p r}+  x^m\Gamma_{jm}^ie(dx^j)i_{\frac{\p}{\p x^i}}.$
Inserting this expression into  (\ref{price}) and multiplying by $e^{\lambda s}$ gives 
\begin{equation}\label{price2}\frac{d}{ds}(e^{\lambda s}s^{p -\frac{n}{2}}\int   y  |z|^2dv)    =  e^{\lambda s}s^{p -\frac{n}{2}}\int (   2\frac{d y}{ds} | i_{\frac{\p}{\p r}}z|^2+ y \langle z,(\lambda-\frac{w}{s} )z\rangle )dv.\end{equation}
Integrate this equality in $s$ from $\sigma^2$ to $\tau^2$ to get 
\begin{multline}\label{priceclub}  e^{\lambda \tau^2}\tau^{2p -n} \int   y(\frac{r^2}{2\tau^2} ) |z|^2dv -
e^{\lambda \sigma^2}\sigma^{2p -n}\int   y(\frac{r^2}{2\sigma^2})  |z|^2dv  \\   =   \int_{\sigma^2}^{\tau^2}e^{\lambda s}s^{p -\frac{n}{2}}\int (   2\frac{d y}{ds} | i_{\frac{\p}{\p r}}z|^2+ y \langle z,(\lambda-\frac{w}{s} )z\rangle )dvds.\end{multline}
Replace $y$ with a sequence $\{y_n\}$ of $C^1$  monotone decreasing functions, supported in $[0,1]$ and converging to the characteristic function of $[0,\frac{1}{2}]$.  Using  the support of $y_n$ and the fact that $|w|= O(r^2)$, we see that we can pick $\lambda>0$ depending on geometric data so that $\lambda > \frac{w}{s}$ on the support of $y$. Hence, the right hand side of (\ref{priceclub}) is nonnegative, yielding 
\begin{equation}\label{pricemono}  e^{\lambda \tau^2}\tau^{2p -n} \int_{B_{\tau}}|z|^2dv \geq 
e^{\lambda \sigma^2}\sigma^{2p -n}\int_{B_\sigma}   |z|^2dv   .\end{equation}
This monotonicity relation immediately implies $z\in L^{2,n-2p}$. 
\end{proof}
 
In the study of harmonic maps and Yang Mills connections, monotonicity relations such as (\ref{pricemono}) are one of the main ingredients in establishing regularity results. 
 
\section{Euler Lagrange Equations}
Write 
$$P(z) = (p^1(z),\cdots p^D(z)),$$
with $p^k(z)$ defined as in Theorem \ref{exists}, still assumed to be quadratic.  Let $V = P^{-1}(0)$. 
Let $z(t)=h+db(t)$ be a curve in $V$, with $b(t)$ a $C^1$ curve in $H_1$. In components, write $z_i(t) = h_i + db_i(t)$, with $h_i$ harmonic and $d^*b_i=0$. Then  
$$0 = \frac{d}{dt}P(z(t)) = DP_zd\dot b = 0,$$
where $DP_z$ denotes the derivative of $P$ at $z$. 
Let $K_z$ denote the kernel of $DP_zd$ in $H_1$. 
\begin{define}
We call $K_z$ the {\em formal tangent space} to $V$ at $z$. 
We say $V$ is {\em unobstructed} at $z$ if for each $v\in K_z$, there is a differentiable curve 
$c(t)$ in $V$ so that $c(0) = z$ and $\dot c(0) = v$.
\end{define}
For simplicity, we now assume  that $\dim M = 4 $ and that the $p^k_{ij}$ are closed.  
Then since $H_1\subset L^4 $ and $z\in L_2$, we may view 
$DP_z$ as a bounded linear map from $H_1\to H_{-1}.$ In the quadratic case under consideration, this map is an algebraic operator given by the adjoint of left multiplication by the product of smooth forms and components of $z$. For example, when 
$p(z) = z\wedge z$, $DP_z = 2e^*(z)$. 
Hence $dDP_z$ defines a bounded linear map from $H_1\to H_{-2}.$
Here we have used $dp^k_{ij} = 0$ to substitute $dDP_z$ for $ DP_zd$. We let  
$d^*$ and $ DP_z^*$ denote the formal $L_2$ adjoint of these two operators, and $(dDP_z)^{\ast_H} = (1+\Delta)^{-1} DP_z^*d^*(1+\Delta)^{-2}$ denotes the Hilbert space adjoint of $dDP_z:H_1\to H_{-2}$. Then  
\begin{equation}\label{perp}K_z = (\text{Im }(dDP_z)^{\ast_H})^{\perp_{H_1}} = ((1+\Delta)^{-1} DP_z^*d^* H_{2})^{\perp_{H_1}}.\end{equation}
Similarly, if  $E(h+db):=\|h+db\|^2$, then
$$\frac{d}{dt}E(b(t)) = \sum_i\langle  d^*z_i, \dot b_i\rangle_{L_2} =  \sum_i\langle (1+\Delta)^{-1}d^*z_i, \dot b_i\rangle_{H_1}:= DE_{b(t)}\dot b.$$
If $z  = h+db $ is an $E$ minimizer, then 
$$DE_{b }v = 0, \,\,\forall v\in K_z,$$
such that there exists a differentiable  curve  $\beta(t)$  in $V$ with $b =\beta(0) $ and $v = \dot\beta(0).$
 
Thus, if $z= h+db$ minimizes $E$ {\em and }$K_z$ {\em is unobstructed}, then the Euler Lagrange equation becomes
\begin{equation} d^*z = \lim_{j\to \infty}DP_z^*d^*\mu_j,\text{ in }H_{-1},\end{equation}
for some  sequence $\mu_j$  in $H_2$. When the image is closed, we can, of course, replace the limit with 
$d^*z = DP_z^*d^*\mu,$ some $\mu \in H_2$. 

Determining when the tangent space is unobstructed is a basic aspect of the regularity theory of nonlinear harmonic forms. We first examine the information following from diffeomorphism invariance of the constraints. 

\section{Diffeomorphism Variations}
In this section we explore first and second variation formulas associated with variation by diffeomorphism. The resulting first variation formula is standard and underlies Price's inequality. The second variation formulas are perhaps less familiar.

If $z\wedge z = 0$, or more generally, if the constraint $p(z)$ satisfies $\phi_t^*p(z) = p(\phi_t^*z),$ for the smooth one parameter family of diffeomorphism $\phi_t$ with $\phi_0(x) = x$, then $\phi_t^*z$ defines a curve in the constraint space, which we can use to verify some of  the formal Euler Lagrange equations. 
Observe that when $z$ is smooth and nonvanishing, then locally there exist coordinates so that if $z\wedge z = 0$ then 
$$z = dx^1\wedge dx^2,$$ and smooth variations of $z$ are given by variations of the functions $x^1$ and $x^2$. 
In particular, the variation is given by the Lie derivative of $z$ by the vectorfield 
$$X:= \dot x^1 \frac{\p}{\p x^1}-\dot x^2 \frac{\p}{\p x^2}.$$
So, we see that in this instance not much information is lost in considering only variations by diffeomorphisms. 
Let $z$ be a minimizer for  $\|\phi_t^*z\|^2$, for all one parameter families of diffeomorphisms $\phi_t$ generated by a vectorfield. Then we have 
\begin{equation}\|z\|^2\leq \|\phi_t^*z\|^2 = \int_{M}z\wedge \phi_{-t}^*(\ast\phi_t^*z).
\end{equation} 
Hence we deduce a second variation formula (and the first variation formula leading to Price's inequality) by expanding 
$$\ast_t:= \phi_{-t}^*\ast\phi_t^* $$ 
to second order in $t$. Express $\ast$ as , $$ \ast=\pm c(dvol),$$
where, $c(dvol)$ denotes clifford multiplication by the volume form, and the sign depends on dimension and degree.    
In a local oriented orthonormal frame $\{v_i\}_{i=1}^n$ and dual coframe $\{\omega^i\}_{i=1}^n$, we write 
$$c(dvol) = c(\omega^1)\cdots c(\omega^n)=(e(\omega^1)-i_{v_1})\cdots(e(\omega^n)-i_{v_n}),$$
where $i_v$ denotes interior multiplication by $v$. It is also convenient to define 
$$\hat c(\omega^i) = e(\omega^i)+i_{v_i}= e(\omega^i)+e^*(\omega^i).$$

Observe that 
$$ \phi_{-t}^* e(\omega^{j })\phi_t^*  = e(\phi_{-t}^*\omega^{j }) ,$$
and 
$$\phi_{-t}^* i_{v_{j }}  \phi_t^*  =   i_{d\phi_{t}v_{j}}  . $$ 
Hence 
\begin{equation}\label{2nd}\begin{split} \phi_{-t}^*(c(\omega^{j })\phi_t^*  &= (e(\phi_{-t}^*\omega^{j })-i_{d\phi_{t}v_{j}})  = c(\omega^j) - t(e(L_X\omega^j)-i_{L_Xv_j})+\frac{t^2}{2}(e(L_X^2\omega^j)-i_{L_X^2v_j})\\
&=:c(\omega^j) - tL_Xc(\omega^j)+\frac{t^2}{2}L_X^2c(\omega^j).\end{split} \end{equation} 
Here $L_X$ denote the Lie derivative with respect to $X$. On differential forms,
\begin{equation}\label{lied}L_X = \{e(\omega^m)\nabla_{v_m},X^ke^*(\omega^k)\} =  e(\omega^m)e^*(\omega^k) X^k_{;m}+\nabla_{X}.\end{equation}
On vector fields, 
\begin{equation}\label{liedv}L_X  = \nabla_{X}  -  X^m_{;k}v_m\otimes\omega^k.\end{equation}
For convenience, fix a point $p$ and choose an orthonormal frame and coframe which is radially covariant constant (in geodesic coordinates) in a neighborhood of $p$. Then we have 
$$\nabla_Y v_j = \frac{1}{2}R(r\frac{\p}{\p r},Y,v_i,v_m)v_m+O(r^2).$$
In particular, $(\nabla_Xv_i)(p) = (\nabla_X^2v_i)(p) = 0.$ In this frame, at $p$,  we have 
\begin{equation}\label{2nd2} \begin{split} \phi_{-t}^* c(\omega^{j })\phi_t^*   & = c(\omega^j) 
- t(e(  X^j_{;m}\omega^m)+i_{X^m_{;j}v_m})\\
&+\frac{t^2}{2}(( X^k_{;m} X^j_{;k}  +  X^k X^j_{;mk})e(\omega^m)+(X^k X^m_{;jk}-  X^m_{;k} X^k_{;j})i_{  v_m   }) +O(t^3)\\
& = c(\omega^j) 
- t(e(  X^j_{;m}\omega^m)+i_{X^m_{;j}v_m})+\frac{t^2}{2}((  (\nabla_{X}X)^j_{;m} - R(X,v_m,v_j,X))e(\omega^m)\\
&+((\nabla_{X}X)^m_{;j} -R(X,v_j,v_m,X)-  2X^m_{;k} X^k_{;j})i_{  v_m   }) +O(t^3).\end{split}  \end{equation} 
We expand this two different ways. Let $A$ and $B$ denote the skew and symmetric summands of $\nabla X$ respectively. 
\begin{equation}\label{2nd2b2} \begin{split}\phi_{-t}^* c(\omega^{j })\phi_t^*&=   c(\omega^j) 
-  t(B^m_{j}\hat c(\omega^m)- A^m_{j} c(\omega^m ))\\
&+\frac{t^2}{2}(( (A^2+B^2)^j_m   +   (X^k A^j_{;m })_{;k}-   X^k_{;k} A^j_{;m})  c(\omega^m) \\
&+ ( \{A,B\}^j_m  +   X^k B^j_{;mk} ) \hat c(\omega^m) ) +O(t^3),\end{split} \end{equation} 
and 
\begin{equation}\label{2nd2bhold} \begin{split}\phi_{-t}^* c(\omega^{j })\phi_t^*&=   c(\omega^j) 
-  t(B^m_{j}\hat c(\omega^m)- A^m_{j} c(\omega^m ))\\
&+\frac{t^2}{2}(( (A^2+B^2+\{A,B\})^m_{;j} -\frac{1}{2}(\nabla_XX)^m_{;j }   +  \frac{1}{2}(\nabla_X X)^j_{;m} ) c(\omega^m) \\
&+ (  -  X^m_{;k} X^k_{;j} + \frac{1}{2}(\nabla_{X}X)^j_{;m}- R(X,v_m,v_j,X)+\frac{1}{2}(\nabla_XX)^m_{;j}) \hat c(\omega^m) ) +O(t^3).\end{split} \end{equation} 
These two expressions lead to two different expressions for $\ast^{-1}\ast_t$. From (\ref{2nd2b2}), we have 
\begin{equation}\label{var2a}\begin{split}\ast^{-1}\ast_t = I+ tB^j_mc(\omega^j)\hat c(\omega^m) +  &\frac{t^2}{2}(-  X^k B^j_{m;k }c(\omega^j)\hat c(\omega^m)+ (\sum_{i,m}B^i_m c(\omega^i)\hat c(\omega^m))^2) +O(t^3).
\end{split}
\end{equation}
From (\ref{2nd2bhold}) we have 
\begin{equation}\label{var2b}\begin{split}\ast^{-1}\ast_t = I+t&B^j_mc(\omega^j)\hat c(\omega^m) + \frac{t^2}{2}([(A^2+B^2)^j_{;m })+R(X,v_m,v_j,X)\\
&-  \frac{1}{2}(\nabla_{X}X)^j_{;m}-\frac{1}{2}(\nabla_XX)^m_{;j}]c(\omega^j)\hat c(\omega^m)+ (\sum_{i,m}B^i_m c(\omega^i)\hat c(\omega^m))^2) +O(t^3).
\end{split}
\end{equation}
These computations are simplified by diagonalizing $B$ and skew diagonalizing $A$ (not simultaneously). It is also useful to note that   the degree preserving summand of $c(\omega^j)\hat c(\omega^m)$ satisfies
$$c(\omega^j)\hat c(\omega^m) = c(\omega^m)\hat c(\omega^j).$$

The $O(t)$ term in either (\ref{var2a}) or (\ref{var2b}) yields the first variation formula. 

\begin{proposition}\label{1varf}First variation formula. Let $\|\phi_t^*z\|^2$ have a critical point at $t=0$ for all smooth 1 parameter families of diffeomorphisms generated by vector fields $X.$ Then 
\begin{equation}\label{1var}0= \langle z,(X^j_{;m}+X^m_{;j})c(\omega^j)\hat c(\omega^m)z\rangle_{L_2}.
\end{equation}
\end{proposition}

Suppose now $z$ is a differential $k-$form satisfying (\ref{1var}). Choose a frame $\{v_j\}_j$  at $p$ such that the bilinear form $Q(v,w):=\langle i_vz,i_wz\rangle$ is diagonal. 
In this frame, we have 
\begin{equation}\label{niceframe}\langle i_{v_j}z,i_{v_m}z\rangle(p) = 0, \text{ for }j\not = m.\end{equation}
In such a frame, (\ref{1var}) becomes 
\begin{equation}\label{f1var}0= 2\int \sum_jX^j_{;j}  |i_{v_j}z|^2dv - \int div(X)|z|^2dv.
\end{equation}
Integration by parts yields 
\begin{equation}\label{fp1var}0=   \langle X, \sum_j  2v_j(|i_{v_j}z|^2)v_j-\nabla|z|^2\rangle
-2\langle  X,\sum_{j,m}\gamma^j_{mm}  (|i_{v_j}z|^2-|i_{v_m}z|^2)v_j\rangle  
.\end{equation}
Hence (no sum over $j$) 
\begin{equation}\label{fp1vareq} v_j( |i_{v_j}z|^2-\frac{1}{2}|z|^2)  = \sum_{ m}\gamma^j_{mm}  (|i_{v_j}z|^2-|i_{v_m}z|^2).\end{equation}
In dimension $4$,  $Q$ is a scalar multiple of the metric at $p$ if and only if $z(p)$ is self dual or antiself dual. Otherwise $Q$ has 2 distinct eigenvalues at $p$. In higher dimensions, the strong assumption that $Q$ has at most two distinct eigenvalues at each point leads to new Bochner type formulas analogous to (\ref{exboch}), except they involve divergences of vectorfields which are not gradient vectorfields. This eigenvalue condition is not diffeomorphism invariant; so, we see no immediate application of this observation.

In dimension $4$ for $z$ a $2-$form, (\ref{fp1vareq}) is equivalent to (\ref{1hflow}) or (\ref{hflow}). 
Thus, in this case, the form satisfies half of the equations given by $dz=0$ and the $d^*z$ equation. The assumption $dz=0$ gives 2 more equations. We therefore verify half the Euler Lagrange equations formally satisfied by this type of nonlinear harmonic, including all the equations needed for the Bochner formula (\ref{nboch}) to hold weakly. If the constraint equations are not diffeomorphism invariant, then the allowed vectorfields $X$ generating diffeomorphisms in the first variation must be correspondingly constrained. For example, if the constraint equation coefficients are functions (only) of a symplectic form, then the first variation formula holds for  locally hamiltonian vectorfields.  If we  restrict  to gradient variations only,  the first variation formula reduces to a single second order equation.  

\begin{proposition}First variation formula for gradient flow minimizing 2-forms in a compact 4-manifold. Let $z$ be a 2-form in a compact oriented 4-manifold such that $\|\phi_t^*z\|^2$ has a critical point at $t=0$ for all smooth 1 parameter families of diffeomorphisms (with $\phi_0 =\text{identity}$) generated by gradient vector fields $X.$ Write $z = aw^1\wedge w^2+bw^3\wedge w^4$ as in Section \ref{gframe}. Then 
\begin{equation}\label{biwave}0=  (\Delta   -2\Delta_{D}^*)(a^2-b^2). 
\end{equation}
\end{proposition}
Similarly, one obtains a single second order equation in the case of locally Hamiltonian vector fields.

The second variation inequality now arises from evaluating the $t^2$ terms in the two expressions for $\langle z,\ast^{-1}\ast_tz\rangle$, when $z$ satisfies the first variation equality, and using the first variation equality to simplify.  We give 2 expressions for the resulting second variation inequality.

\begin{proposition}Second variation formula. Let $\|\phi_t^*z\|^2$ have a local minimum at $t=0$ for all smooth 1 parameter families of diffeomorphisms, $\phi_t$ generated by gradient fields $X.$ Then 
\begin{equation}\label{vart1} 0\leq \langle z, (   - X^k B^j_{m ;k} c(\omega^j)\hat c(\omega^m)+(\sum_{i,m}B^i_m c(\omega^i)\hat c(\omega^m) )^2)z\rangle_{L_2},
\end{equation}
and 
\begin{equation}\label{vart2} 0\leq \langle z,(  [(A^2+B^2)^j_{m })+R(X,v_m,v_j,X)]c(\omega^j)\hat c(\omega^m)+(\sum_{i,m}B^i_m c(\omega^i)\hat c(\omega^m) )^2)z\rangle_{L_2}.
 \end{equation}
\end{proposition}
Observe that these inequalities also apply to unconstrained harmonic forms. In this case, (\ref{vart2}) gives  vanishing theorems when we have a sufficient supply of vector fields with controlled covariant derivative. For example, Grassmann manifolds viewed as spaces of unitary projections have the functions $F_C(p) = \tr p\,C ,$
for each hermitian matrix $C$. (See for example \cite[Theorem 1.4.12]{Xin}.) Using $X=\nabla F_A$ in the second variation formula, with $A$ running over an orthonormal basis of hermitian matrices, one can obtain numerous cohomology vanishing theorems, which, however, can be proved more easily by other techniques.

Orientation preserving isometries preserve the Hodge star operator. Hence for $X$ a Killing field, the second variation should be trivial. For Killing fields, $B=0$.  More generally, one expects the $A$ term to drop out of the second variation formula, as it does in (\ref{vart1}) but not (\ref{vart2}). This suggests (\ref{vart2}) might contain information about Killing vectorfields.

Suppose that $X$ is a Killing vectorfield. Then $B=0$, and the second variation inequality becomes an equality, yielding  
\begin{equation} 0= \int  \sum_j(|\nabla_jX|^2  - R(X,v_j,v_j,X)) (2|i_{v_j}z|^2 -|z|^2)dv,
\end{equation}
generalizing the integral form of the  Bochner formula for Killing fields. Choosing $z=dvol$ returns the usual  formula.

\end{document}